\documentclass[12pt]{article}
\usepackage{amsthm}
\usepackage{amsmath}
\usepackage{verbatim,fullpage}
\newtheorem{theorem}{Theorem}[section]
\newtheorem{conjecture}[theorem]{Conjecture}
\newtheorem{claim}[theorem]{Claim}
\newtheorem{lemma}[theorem]{Lemma}
\def\ERT{P.~Erd\H{o}s, A.L.\ Rubin, and H.\ Taylor}
\def\ceil#1{\lceil #1 \rceil}
\begin{document}
%\pagewiselinenumbers
\title{Choice number of complete multipartite graphs
$K_{3*3,2*(k-5),1*2}$ and $K_{4,3*2,2*(k-6),1*3}$}
\author{
Wenjie He$^a$,
Lingmin Zhang$^b$,
Daniel W. Cranston$^c$,
Yufa Shen$^b$,
Guoping Zheng$^b$}
\footnotetext{$^*$Research supported in part by the Doctoral Foundation of Hebei Normal University of Science and Technology, P.R.\ China, and the Science Foundation of the Education Department of Hebei Province, P.R.\ China (Grant No. 2005108).
\smallskip

\textit{Email:} he\_wenjie@yahoo.com (Wenjie He).
\smallskip

\textit{Email:} lingmin9999@163.com (Lingmin Zhang).
\smallskip

\textit{Email:} dcransto@dimacs.rutgers.edu (Daniel W. Cranston).
\smallskip

\textit{Email:} syf030514@163.com (Yufa Shen).
}
\date{}
\maketitle
%\vspace{-.3in}
\begin{center}
\textit{
$^a$Applied Mathematics Institute, Hebei University of Technology,\\
Tianjin 300130, P.R. China \\
}
\bigskip

\textit{
$^b$Dept. of Math. and Phys., Hebei Normal University of Science and Technology,\\
Qinhuangdao 066004, P.R. China
}
\bigskip

\textit{$^c$DIMACS, Rutgers University,\\
Piscataway, NJ 08854, USA
}
\end{center}
\abstract{
A graph $G$ is called \textit{chromatic-choosable} if its choice number is equal
to its chromatic number, namely $Ch(G)=\chi(G)$.  Ohba has conjectured that every
graph $G$ satisfying $|V(G)|\leq 2\chi(G)+1$ is chromatic-choosable.
Since each $k$-chromatic graph is a subgraph of a complete $k$-partite graph, we
see that Ohba's conjecture is true if and only if it is true for every complete
multipartite graph.  However, the only complete multipartite graphs for which
Ohba's conjecture has been verified are:
$K_{3*2,2*(k-3),1}$,
$K_{3,2*(k-1)}$,
$K_{s+3,2*(k-s-1),1*s}$,
$K_{4,3,2*(k-4),1*2}$, and
$K_{5,3,2*(k-5),1*3}$.
In this paper, we show that Ohba's conjecture is true for two new classes
of complete multipartite graphs: graphs with three parts of size 3 and graphs with
one part of size 4 and two parts of size 3.
Namely, we prove that
%Ohba's conjecture for
$Ch(K_{3*3,2*(k-5),1*2})=k$ and
$Ch(K_{4,3*2,2*(k-6),1*3})=k$ (for $k\geq 5$ and $k\geq 6$, respectively).
}
\bigskip

\bigskip

\bigskip

\section{Introduction}
\label{intro}
\textit{List colorings} of graphs were introduced independently by V.G. Vizing~\cite{vizing} and \ERT~\cite{erdos}.
For a graph $G=(V,E)$ and each vertex $u\in V(G)$, let $L(u)$ denote a list of colors available for $u$.  We call $L=\{L(u)|u\in V(G)\}$ a \textit{list assignment} of $G$,
and we call $L$ a \textit{$k$-list assignment}  if $|L(u)|\geq k$ for each $u\in V(G)$.
An \textit{$L$-coloring} of $G$ is a coloring in which each vertex receives a color from its own list such that adjacent vertices get different colors.
A graph $G$ is called \textit{$k$-choosable} if $G$ is $L$-colorable for every $k$-list assignment $L$.  The \textit{chromatic number} $\chi(G)$ of $G$ is the smallest integer $k$ such that $G$ is $k$-colorable, and the \textit{choice number} $Ch(G)$ of a graph $G$ is the
smallest integer $k$ such that $G$ is $k$-choosable.

It is easy to see that every graph $G$ satisfies $Ch(G)\geq\chi(G)$.
However, the equality does not necessarily hold.
In fact, Erd\H{o}s et al.~\cite{erdos} showed that there are bipartite graphs
with choice numbers that are arbitrarily large.
If a graph $G$ satisfies $Ch(G)=\chi(G)$, then $G$ is called \textit{chromatic-choosable}.
Much of the work on choice number has studied graph classes in which every
graph is chromatic-choosable.
The most famous conjecture in this area is the \textit{List Coloring Conjecture} (see~\cite{haggkvist}),
which states that every line graph is chromatic-choosable.
Galvin~\cite{galvin} proved the special case that every line graph of a bipartite
graph is chromatic-choosable (for more information about chromatic-choosability, we refer the reader to a survey by Woodall~\cite{woodall}).
In this paper, we focus our attention on Ohba's conjecture:

\begin{conjecture}[Ohba~\cite{ohba}]
\label{conj1}
If $|V(G)|\leq 2\chi(G)+1$, then $Ch(G)=\chi(G)$.
\end{conjecture}

Because every $k$-chromatic graph is a subgraph of a complete $k$-partite graph,
Ohba's conjecture is true if and only if it is true for every complete
multipartite graph.
Thus, we can restate Ohba's conjecture in the following way.

\begin{conjecture}
\label{conj2}
If $G$ is a complete $k$-partite graph with $|V(G)|=2k+1$, then $Ch(G)=\chi(G)=k$.
\end{conjecture}

As a result of the formulation in Conjecture~\ref{conj2}, all of the
work done on proving Ohba's conjecture has focused on proving it for
specific classes of complete multipartite graphs.  We use the
notation $K_{r*s}$ to denote a complete $s$-partite graph in which
each part has $r$ vertices. Analagously, we  use  the notation
$K_{r*s,t*u}$ to denote a complete $(s+u)$-partite graph, in which
$s$ parts have $r$ vertices and $u$ parts have $t$ vertices. Here we
list the complete multipartite graphs for which the choice number is
known.

\begin{theorem}[\cite{erdos}]
\label{first}
$Ch(K_{2*k})=k.$
\end{theorem}

\begin{theorem}[\cite{kierstead}]
$Ch(K_{3*k})=\ceil{\frac{4k-1}{3}}.$
\end{theorem}

\begin{theorem}[\cite{ohba2}]
$Ch(K_{3*r,1*t})=\max(r+t,\ceil{\frac{4r+2t-1}{3}}.$
\end{theorem}

\begin{theorem}[\cite{gravier}]
If $k\geq 3$, then $Ch(K_{3*2,2*(k-2)})=k.$
\end{theorem}

\begin{theorem}[\cite{enomoto}]
\label{useful}
$Ch(K_{4,2*(k-1)})=
\begin{cases}
k, & \mbox{if $k$ is odd} \\
k+1, &\mbox{if $k+1$ is even.}
\end{cases}
$
\end{theorem}

\begin{theorem}[\cite{enomoto}]
If $m\leq 2s+1$, then $Ch(K_{m,2*(k-s-1),1*s})=k.$
\end{theorem}

\begin{theorem}[\cite{shen},\cite{shen2}]
\label{last}
$Ch(K_{s+2,3,2*(k-s-2),1*s})=k$ for $s\in\{2,3,4\}$.
\end{theorem}

To obtain partial results for Ohba's conjecture from
Theorems~\ref{first} through~\ref{last},
we consider subgraphs of the graphs studied in
the seven theorems above
that are $k$-partite and have $2k+1$ vertices.
In particular, we conclude that the choice number is $k$ for every multipartite
graph of the following form:
$K_{3*2,2*(k-3),1}$,
%$K_{3,2*(k-1)}$,
$K_{4,2*(k-2),1}$,
$K_{s+3,2*(k-s-1),1*s}$ (for all $s$), and
%$K_{4,3,2*(k-4),1*2}$, and
%$K_{5,3,2*(k-5),1*3}$.
$K_{s+2,3,2*(k-s-2),1*s}$ (for $s\leq 4$).
For $K_{4,2*(k-2),1}$, if $k$ is odd, then the result follows directly from
Theorem~\ref{useful}.  If $k$ is even, then we first color the vertex $v$ in the
unique part of size 1.  Since the remaining graph $G-v$ is $(k-1)$-choosable by
Theorem~\ref{useful}, we see that $G$ is $k$-choosable.

\section{Preliminaries}
In this section we introduce three tools that significantly reduce the number
of cases we must consider in each of our proofs.

For a graph $G=(V,E)$ and a subset $X\subseteq V$, let $G[X]$ denote the
subgraph of $G$ induced by $X$.  For a list assignment $L$ of $G$, let $L|_X$ denote
$L$ restricted to $X$, and let $L(X)$ denote the union $\cup_{u\in X}L(u)$.
If $A$ is a set of colors, let $L\backslash A$ denote the list assignment obtained
from $L$ by deleting the colors in $A$ from each $L(u)$ with $u\in V(G)$.
When $A$ consists of a single color $a$, we write $L-a$ instead of $L\backslash\{a\}$.

We say that a graph $G$ satisfies \textit{Hall's condition} for a list assignment $L$
if $|L(X)|\geq |X|$ for every subset $X\subseteq V(G)$.
A result of Hall implies the following theorem (this is commonly called Hall's Theorem):
\begin{theorem}[Hall]
\label{hall}
If $G$ satisfies Hall's condition, then there exists an
$L$-coloring of $G$ in which each vertex receives a distinct color.
\end{theorem}
%\noindent
Kierstead~\cite{kierstead} used Theorem~\ref{hall} to prove the following lemma.
This result will be of great use to us.

\begin{lemma}[\cite{kierstead}] Let $L$ be a list assignment for a graph $G=(V,E)$
\label{lemma1}
and let $X\subseteq V(G)$ be a maximal non-empty subset such that $|L(X)|<|X|$.
If $G[X]$ is $L|_X$-colorable, then graph $G$ is $L$-colorable.
\end{lemma}
\begin{proof}
Let $X$ be a maximal subset of $V$ such that either $X=\emptyset$ or $|L(X)|<|X|$.
Let $C=L(X)$.  By the maximality of $X$, every subset $Y\subset V\backslash X$ satisifes
$|L(Y)\backslash C|\geq |Y|$.  Let $L'(v)=L(v)\backslash C$ for every $v\in V\backslash X$.
Note that $G\backslash V$ satisfies Hall's condition for $L'$.  Hence $G\backslash V$ is $L'$-colorable.
By hypothesis, $X$ is $L|_X$-colorable.
Since none of the colors used on $X$ are used on $V\backslash X$, we can combine the two
colorings to give an $L$-coloring of $G$.
\end{proof}

Kierstead used Lemma~\ref{lemma1} to prove the following lemma.

\begin{lemma}[\cite{kierstead}]
\label{lemma2}
A graph $G=(V,E)$ is $k$-choosable if $G$ is $L$-colorable for every $k$-list assignment $L$ such that $|L(V)| < |V|$.
\end{lemma}

We would like to apply Lemma~\ref{lemma2} in the middle of constructing a coloring.
However, at that point the number of colors available at one vertex may be different
from the number of colors available at another vertex.  Thus we will prove a
more general version of Kierstead's second lemma, which will apply even when different
vertices may have lists of different sizes.
We need the following definition.
Let $L$ be a list assignment.
We say $G$ is \textit{$L$-size-choosable} if $G$ is $L_1$-colorable for every list
assignment $L_1$ such that $|L_1(v)|=|L(v)|$ for each $v\in V(G)$.  This is a generalization of $k$-choosability, since distinct vertices may have lists of different sizes.

Now we can state Lemma~\ref{keylemma}, which is a generalization of Lemma~\ref{lemma2} to the case where distinct vertices may have lists of different sizes.
Our proof of Lemma~\ref{keylemma} is essentially the same as Kierstead's proof of
Lemma~\ref{lemma2}.

\begin{lemma}
\label{keylemma}
%\label{lemma3}
Let L be a list assignment such that $|L(v)| < |V|$ for each $v\in V$.
A graph $G=(V,E)$ is $L$-size-choosable if $G$ is $L_1$-colorable for every list assignment $L_1$ such that $|L_1(V)| < |V|$ and $|L_1(v)|=|L(v)|$ for each $v\in V$.
\end{lemma}

\begin{proof}
Fix a list assignment $L_{0}$ such that $|L_{0}(v)|=|L(v)|$ for each
$v\in V$. Suppose $G$ is $L_{1}$-colorable for every list assignment
$L_{1}$ such that $|L_{1}(V)|<|V|$  and $|L_{1}(v)|=|L(v)|$ for each
$v\in V$.
\par We show that the hypothesis of Lemma 2.1 holds for $G$ and
$L_{0}$. Consider any maximal  non-empty subset $X\subset V$ such
that $|L_{0}(X)|<|X|$. Let $A=L_{0}(X)$. Choose $u\in V-X$ such that
$|L_{0}(u)|\geq |L_{0}(w)|$ for all $w\in V-X$. We define a list
assignment $L_{1}$. We consider two cases depending on whether
$|L_{0}(u)|\leq|A|$ or not.
 \par If $|L_{0}(u)|\leq|A|$, then we
define $L_{1}$ as follows.  If $v\in X$, then $L_{1}(v)=L_{0}(v)$.
If $v\not\in X$, then $L_{2}(v)$ is a subset of $A$ of size
$|L_{0}(v)|$.
 \par If $|L_{0}(u)|>|A|$, then we define $L_{1}$  as follows. Let $B$ be
a subset of $L_{0}(u)\backslash A$ of size $|L_{0}(u)|-|A|$. If
$v\in X$, then $L_{1}(v)=L_{0}(v)$. If $v\not\in X$, then $L_{1}(v)$
is a subset of $A\cup B$ of size $|L_{0}(v)|$.
\par In the first
case, $|L_{1}(V)|=|L_{0}(X)|<|X|\leq|V|$.  In the second case,
$|L_{1}(V)|=|A\cup B|=|L_{0}(u)|<|V|$.  By hypothesis, $G$
 is $L_{1}$-colorable. Since $L_{0}|_{X}=L_{1}|_{X}$, we see that $G[X]$
  is $L_{0}|_{X}$-colorable. Thus, by Lemma $2.2$, $G$ is
  $L_{0}$-colorable.
\end{proof}

In the process of constructing a coloring, we will repeatedly choose a color to
use on 2 or 3 vertices, then delete that color from the lists of colors available
at each uncolored vertex.
We must then show that we can color the remaining uncolored
vertices from their lists.  Each time we choose a color to use on one or more vertices, Lemma~\ref{keylemma} enables us to assume that the total number of colors
available on the uncolored vertices is smaller than the number of uncolored vertices.
We use this technique frequently in the proofs in Sections~\ref{sec1} and~\ref{sec2}.

\section{Ohba's conjecture is true for $K_{4,3*2,2*(k-6),1*3}$}
\label{sec1}
We are now ready to prove our first main theorem.
In Sections~\ref{sec1} and \ref{sec2}, we will often conclude a case in the
analysis by saying
that we can finish by coloring greedily.  By this we mean that we can color the uncolored vertices greedily in order of nondecreasing list size.
Frequently we will use phrases like ``there exists some vertex in $X$, say $x_1$, such that color $c_1\in L(x_1)$''; by this we mean that without loss of generality we may assume that the desired vertex is $x_1$.
%\newpage
\begin{theorem}
\label{theorem1}
If $G=K_{4,3*2,2*(k-6),1*3}$, then $Ch(G)=k$.
\end{theorem}
\begin{proof}
Let $G=K_{4,3*2,2*(k-6),1*3}$.
We label the parts of sizes 4, 3, and 1 as follows: $X=\{x_1,x_2,x_3,x_4\}$, $Y=\{y_1,y_2,y_3\}$, $Z=\{z_1,z_2,z_3\}$, $W_1=\{w_1\}$, $W_2=\{w_2\}$, and $W_3=\{w_3\}$ (we do not label the parts of size 2 since they will be less important in the argument).

We begin by handling the case when all the vertices in a part of size 3 or 4 have a common color.  Clearly, we should use this common color on all the vertices in the part.
Intuitively, this case should be easier than the general case.  In fact, the analysis is straightforward.  However, for brevity, we observe that the remaining uncolored vertices will form a subgraph of $K_{4,3,2*(k-5),1*2}$ and recall that Shen et al.~\cite{shen} proved
that $Ch(K_{4,3,2*(k-5),1*2})=k-1$.  So for the rest of this proof, we assume that no
color appears on all the vertices in a part of size 3 or 4.

Similarly, if the 2 vertices in a part of size 2 have a common color, we will use
the common color on both of them.  To formalize this, we induct on the number of
parts of size 2\ in which the vertices have a common color.
The induction step is easy.  We use the common color on each vertex in the part (of size 2), remove that color from the lists of all other vertices, then
recurse on the graph with both vertices of that part deleted.
Hence, for our base case, we assume that no color appears on both
vertices in a part of size 2.

We first consider the case when no color appears on more than 2 vertices in $X$ (later, we will consider the case when a color appears on 3 vertices in $X$ and
show that that case reduces to the present case).

By Lemma~\ref{keylemma}, we can assume that some color, say $c_1$, appears on two vertices in $Y$, say $y_1$ and $y_2$.  Use $c_1$ on $y_1$ and $y_2$.
Let $L(v)$ denote the list of available colors at each vertex $v$ after we have
used color $c_1$ on $y_1$ and $y_2$.
Now let $U$ be a maximal subset of uncolored vertices $U\subset V(G)$ such that $|L(U)|<|U|$.
Note that $|L(x_1)|+|L(x_2)|+|L(x_3)|+|L(x_4)|\geq 4k-2$.
Since no color appears on three vertices of $X$,
we have $|L(x_1)\cup L(x_2)\cup L(x_3)\cup L(x_4)|\geq (4k-2)/2 = 2k-1\geq |U|$.
Hence, $U$ contains at most 3 vertices from $X$; call these $x_1$, $x_2$, and $x_3$.
Since each pair of vertices in the same part of size 2 have disjoint lists,
each part of size 2 contains at most 1 vertex of $U$.
Since each vertex in a part of size 2 has at least $k-1$
colors available, we can greedily color the vertices of $U$ in parts of size 2.
Since there are only $k-6$ parts of size 2, each vertex loses at most $k-6$ colors.
So we have reduced our problem to coloring the vertices of $U$ in parts of sizes 1, 3,
and 4.
Let $L'(v)$ denote the list of available colors at each uncolored vertex $v\in U$ after we have
colored all the vertices of $U$ in parts of size 2.
We have the following inequalities: $|L'(w_1)|\geq 5$,
$|L'(w_2)|\geq 5$, $|L'(w_3)|\geq 5$ and $|L'(y_3)|\geq 6$.
Wlog, we also have the inequalities: $|L'(x_1)|\geq 6$, $|L'(x_2)|\geq 5$,
$|L'(x_3)|\geq 5$, $|L'(z_1)|\geq 6$, $|L'(z_2)|\geq 5$, and $|L'(z_3)|\geq 5$.
We assume that each inequality holds with equality.
Let $U'$ denote this set of 10 vertices.
The set $U$ may not contain all of $U'$,
but if we can color the graph $G[U']$, that will imply that $G$ is $L$-colorable.

At this point, we observe that the case when 3 vertices of $X$ have a common color reduces to the present case.
In that case we use the common color on the three vertices on which it appears.
By the same analysis as above, we again reduce the problem to coloring the vertices of $U$ that are in parts of sizes 1, 3, and 4.  In that case $U$ contains at most 3 vertices
of $Y$ and at most 1 vertex of $X$.  By relabeling the vertices of $Y$ as $x_1$, $x_2$,
and $x_3$ and relabeling vertex $x_1$ as $y_1$, we reach the present situation.
Each of the inequalities given above still holds.

Let $A=L'(y_3)\cup L'(w_1)\cup L'(w_2)\cup L'(w_3)$.  We consider two cases: $|A|\geq 7$ and
$|A|=6$.
\bigskip

\noindent
\textbf{Case 1:} $|A|\geq 7$. \\
Since $|U'|=10$, by Lemma~\ref{keylemma}, we may assume that $|L'(U')|\leq 9$.
Since $|L'(x_1)|+|L'(x_2)|+|L'(x_3)|\geq 16$, at least $16-9=7$ colors each appear on two vertices in $X$ (since no color appears on all three vertices of $X$).
So we can choose some color $c_2$ that appears on two vertices in $X$, say $x_1$ and $x_2$, such that $c_2\notin L'(z_1)$.
Use color $c_2$ on vertices $x_1$ and $x_2$.
Let $L''(v)=L'(v)-c_2$ for each uncolored vertex $v\in U'$.
Since $|L''(z_2)|+|L''(z_3)|\geq 8$, by Lemma~\ref{keylemma} we may assume that
vertices $z_2$ and $z_3$
must have a common color, call it $c_3$.  Use color $c_3$ on vertices $z_2$ and $z_3$.
Let $L'''(v)$ denote the lists of remaining colors for each vertex $v\in U'\backslash
\{z_2,z_3\}$.
We have the following inequalities:
$|L'''(x_3)|\geq 4$,
$|L'''(y_3)|\geq 4$,
$|L'''(z_1)|\geq 6$,
$|L'''(w_1)|\geq 3$,
$|L'''(w_2)|\geq 3$,
$|L'''(w_3)|\geq 3$,
and $|L''(y_3)\cup L''(w_1) \cup L''(w_2) \cup L''(w_3)|\geq 5$.
It is easy to verify that Hall's condition holds.  Hence, $G$ is $L$-colorable.
\bigskip

\noindent
\textbf{Case 2:} $|A|=6$. \\
Since $|U'|=10$, by Lemma~\ref{keylemma}, we may assume that $|L'(U')|\leq 9$.
Since $|L'(x_1)|+|L'(x_2)|+|L'(x_3)|\geq 16$, at least $16-9=7$ colors appear on two vertices in $X$.
So we can choose some color $c_2$ that appears on two vertices in $X$, say $x_1$ and $x_2$, such that $c_2\notin A$.  Use $c_2$ on vertices $x_1$ and $x_2$.
Let $U''=U'\backslash \{x_1,x_2\}$ and $L''(v)=L'(v)-c_2$ for each vertex $v\in U''$.
By Lemma~\ref{keylemma}, we may assume that $|L''(U'')| < |U''| = 8$.
Since $|L''(z_1)|+|L''(z_2)|+|L''(z_3)|\geq 14$, we see that $14-7=7$ colors must each appear on two vertices in $Z$.
So we can choose some color $c_3$ that appears on two vertices in $Z$, say $z_1$ and $z_2$, such that $c_3\notin A$.  Use color $c_3$ on vertices $z_1$ and $z_2$.

Let $L'''(v)$ denote the lists of remaining colors for each uncolored vertex $v$.
We have the inequalities:
$|L'''(x_3)|\geq 4$,
$|L'''(y_3)|\geq 6$,
$|L'''(z_3)|\geq 4$,
$|L'''(w_1)|\geq 5$,
$|L'''(w_2)|\geq 5$, and
$|L'''(w_3)|\geq 5$.
We can finish by coloring greedily.
Hence, $G$ is $L$-colorable.
\end{proof}

\section{Ohba's conjecture is true for $K_{3*3,2*(k-5),1*2}$}
\label{sec2}
We will now prove our second main theorem.
The proof is similar to the proof of Theorem~\ref{theorem1};
however, the one fewer part of size 1 requires a more
complex argument.

\begin{theorem}
If $G=K_{3*3,2*(k-5),1*2}$, then $Ch(G)=k$.
\label{theorem2}
\end{theorem}
\begin{proof}
It is easy to handle the case when all the vertices in a part of size 2 or size 3
have a common color.
We will use that common color on all the vertices in that part.
To formalize this, we use induction on the number of parts of size 2 or 3
in which all the vertices have a common color.

The induction step is easy.
Let $S$ be a part (of size 2 or 3) in which the vertices have a common color.
We use the common color on each vertex in $S$, remove that color from the lists of all other vertices, then recurse on $G-S$.
If $S$ has size 2, then we recurse on a graph with one fewer part of size 2.
If $S$ has size 3, then we recurse on a proper subgraph of the graph we consider
when $S$ has size 2 (so the claim follows).
Hence, for our base case, we assume that no color appears in the lists of all
the vertices in a part of size 2 or 3.

We label the parts of sizes 3 and 1 as follows:
$X=\{x_1,x_2,x_3\}$,
$Y=\{y_1,y_2,y_3\}$,
$Z=\{z_1,z_2,z_3\}$,
$W_1=\{w_1\}$, and
$W_2=\{w_2\}$
(we do not label the parts of size 2 because they will be less important in the argument).
We would like to find 2 vertices in $X$, say $x_1$ and $x_2$, with a common color,
say $c_1$, and
2 vertices in $Y$, say $y_1$ and $y_2$, with a common color, say $c_2\neq c_1$, such that there exists a
vertex in $Z$, call it $z_1$, such that $\{c_1,c_2\}\cap L(z_1)=\emptyset$.
It will also be fine if part $Z$ is interchanged with part $X$ or part $Y$
in these conditions.  We now show that we can do this.

\begin{claim}
\label{lemmaclaim}
We can find 2 vertices in $X$, say $x_1$ and $x_2$, with a common color,
say $c_1$, and
2 vertices in $Y$, say $y_1$ and $y_2$, with a common color, say $c_2\neq c_1$, such that there exists a
vertex in $Z$, call it $z_1$, such that $\{c_1,c_2\}\cap L(z_1)=\emptyset$.
It will also be fine if part $Z$ is interchanged with part $X$ or part $Y$
in these conditions. \end{claim}
\begin{proof}[Proof of claim~\ref{lemmaclaim}]
By Lemma~\ref{keylemma}, we can assume that $|L(V)|<|V|=2k+1$.
Note that $|L(x_1)|+|L(x_2)|+|L(x_3)|=3k$.  Since $|L(x_1)\cup L(x_2)\cup L(x_3)|\leq |L(V)|\leq 2k$, there are at least $k$ colors that each
show up on at least 2 vertices in $X$; the same is true for parts $Y$ and $Z$.
Recall that $k\geq 5$.
Note that if at least 4 colors each appear on 2 vertices in $X$ and
also each appear on 2 vertices in $Y$, then the claim holds for the following reason.
Each of these 4 colors does not appear on at least 1 vertex of $Z$.  Since there
are 3 vertices in $Z$, 2 of these 4 colors (call them $c_1$ and $c_2$)
``miss'' the same vertex in $Z$.
So we can use color $c_1$ on 2 vertices of $X$ and use color $c_2$ on 2 vertices of $Y$.
Hence, we may assume that some color that appears on 2 vertices of $X$ must
appear on either 0 or 1 vertices of $Y$; we consider these two cases separately.

Suppose that color $c_1$ appears on 2 vertices in $X$, but that color $c_1$
does not appear on any vertex in $Y$.
Now we can use color $c_1$ on 2 vertices of $X$, and use any choice of $c_2\neq c_1$ on 2 vertices in $Z$.  Hence, we can choose colors $c_1$ and $c_2$ as desired.

Instead suppose that color $c_1$ appears on 2 vertices of $X$, but in $Y$ color $c_1$ only appears on one vertex, say $y_1$.  Recall that at least $k$ colors appear on 2 vertices in $Z$.  Consider the at least $k-1\geq 4$ colors other than $c_1$ that appear on 2 vertices in $Z$.  If one of these colors does not appear at $y_2$ or $y_3$, then the claim holds.
So we may assume that at least 4 of the colors that each appear on two vertices in $Z$
also appear on both $y_2$ and $y_3$.  Again, the claim holds, as in the first
paragraph of the proof.
\end{proof}

Use color $c_1$ on vertices $x_1$ and $x_2$.
Let $G'=G\backslash \{x_1,x_2\}$ and $L'=L\backslash c_1$.
Let $U$ be a maximal nonempty subset $U\subseteq V(G')$ such that $|L(U)|<|U|$.
By Lemma 2.1, $G'$ is $L'$-colorable if $G'[U]$ is $L'|_U$-colorable.
Thus, the remainder of our argument will show that $G'[U]$ is $L'|_U$-colorable.
Note that each part of size 2 has at most one vertex in $U$ (otherwise, $|L(U)|\geq 2k-1 \geq |U|$, since the lists of any two vertices in the same part of size 2  must be disjoint).  Since each vertex in a part of size 2 has at least $k-1$ colors available, we can greedily color all the vertices of $U$ in parts of size 2 without using color $c_2$.
(Note that the size of the list for each vertex decreases by at most $k-5$
since there are only $k-5$ parts of size 2.)
So now we only need to color the vertices of $U$ in parts of sizes 3 and 1.
In fact, we will color all the uncolored vertices (not just those in $U$)
in parts of sizes 3 and 1.
Let $U'$ denote the set of uncolored vertices in parts of sizes 3 and 1.
Let $L''(v)$ denote the lists of colors available at each vertex $v\in U'$
after we have colored all the vertices of $U$ in parts of size
2.  We have the following inequalities:
$|L''(x_3)|\geq 5$,
$|L''(w_1)|\geq 4$,
$|L''(w_2)|\geq 4$,
Without loss of generality, we have the additional inequalitites:
$|L''(y_1)|\geq 5$,
$|L''(y_2)|\geq 4$,
$|L''(y_3)|\geq 4$,
$|L''(z_1)|\geq 5$,
$|L''(z_2)|\geq 4$,
$|L''(z_3)|\geq 4$.
We assume that each of the inequalities holds with equality.
Let $A=L''(x_3)\cup L''(w_1)\cup L''(w_2)$.  We consider two cases: $|A|\geq 6$ and
$|A|=5$.
\bigskip

\noindent
\textbf{Case 1:} $|A|\geq 6$. \\
Use color $c_2$ on vertices $y_1$ and $y_2$.
Note that $|L''(z_1)|\geq 5$ (recall that $c_1\notin L(z_1)$)
and that vertex $z_1$ is only adjacent to 4
uncolored vertices in $G[U']$.  Hence, any coloring of the other 6 uncolored
vertices in $U'$ can be extended to $z_1$.  So let $U''=U'\backslash\{y_1,y_2,z_1\}$.
Now we need to show that $G[U'']$ is $L''|_{U''}$-colorable.
By Lemma~\ref{keylemma}, we may assume that $|L''(U'')| < |U''| = 6$.
Since $|L''(z_2)-c_2|+|L''(z_3)-c_2|\geq 6$, there exists a color $c_3\in L''(z_2)\cap L''(z_3)$;
use $c_3$ on vertices $z_2$ and $z_3$.
Let $U'''=U''\backslash \{z_2,z_3\}$ and
$L'''(v)=L''(v)-\{c_3\}$ for every vertex $v\in U'''$.
Now we have $|L'''(w_1)|\geq 2$, $|L'''(w_2)|\geq 2$, $|L'''(y_3)|\geq 3$,
$|L'''(x_3)|\geq 3$, and $|L'''(x_3)\cup L'''(w_1)\cup L'''(w_2)|\geq 4$.
It is straightforward to verify that the four remaining uncolored vertices satisfy Hall's condition.
As a result, $G[U''']$ is $L'''|_{U'''}$-colorable.  Thus, $G$ is $L$-colorable.
\bigskip

\noindent
\textbf{Case 2:} $|A|=5$. \\
%Let $A=L''(x_3)\cup L''(w_1)\cup L''(w_2)$.
We would like to find two vertices
both in $Y$ (or both in $Z$), call them $y_1$ and $y_2$, such that there exists a color $c_3\in (L''(y_1)\cap L''(y_2))\backslash A$.
(In Claim~\ref{lemmaclaim} we previously specified two vertices to be $y_1$ and $y_2$; now we relabel vertices if necessary.)
We consider two subcases, depending on whether or not we can find such vertices.
\bigskip

\noindent
\textbf{Subcase 2.1:}
There exists $c_3\in(L''(y_1)\cap L''(y_2))\backslash A$. \\
Use $c_3$ on vertices $y_1$ and $y_2$.
Let $U'=U\backslash \{y_1,y_2\}$. and Let $L'''(v)=L''(v)- c_3$ for all $v\in U'$.
Note that $|L'''(z_1)|+|L'''(z_2)|+|L'''(z_3)|\geq 11$.  Since $|A|=5$ and no
color in $A$ appears on all three vertices of $Z$, some vertex in $Z$ has a color
available that is not in $A$.  Wlog, say this is color $c_4$ on vertex $z_1$; use color $c_4$ on $z_1$.
There are 6 remaining uncolored vertices.
By Lemma~\ref{keylemma}, we can assume that the union of the lists for these 6 remaining vertices
has size at most 5.
%$|L(U''-\{z_1\})| < |U''-\{z_1\}| = 6$.
Since $|L'''(z_2)|+|L'''(z_3)|\geq 6$,
there exists $c_5\in L(z_2)\cap L(z_3)$.  After using $c_5$ on vertices $z_2$ and $z_3$,
we can color the four remaining uncolored vertices greedily.  Thus, $G$ is $L$-colorable.
\bigskip

\noindent
\textbf{Subcase 2.2:}
$(L''(y_i)\cap L''(y_j))-A=\emptyset$ for all $i\neq j\in \{1,2,3\}$. \\
By symmetry, we can also assume $(L''(z_i)\cap L''(z_j))-A=\emptyset$
for all $i\neq j\in \{1,2,3\}$.
Since $|L''(y_1)|+|L''(y_2)|+|L''(y_3)|\geq 13>2|A|=10$, there exists some vertex of $Y$, say $y_1$, with $c_4\in L(y_1)-A$.  Use color $c_4$ on $y_1$ and let $U''=U'-y_1$.
(Note that color $c_4$ is available on at most one vertex in $Z$.)
By Lemma~\ref{keylemma}, we may assume that $|L''(U'')-c_4| < |U''|=8$.
Since $|L''(y_2)|+|L''(y_3)|\geq8$, there exists a color $c_5\in L''(y_2)\cap L''(y_3)$.
 Use color $c_5$ on $y_2$ and $y_3$.

Some vertex in $Z$, say $z_1$, has at least 4 available colors.
Note that $z_1$ is only adjacent to 3 uncolored vertices in $U''$.
Hence, any coloring of the other 6 uncolored vertices in $U''$ can be extended to $z_1$.  Let $U'''=U''\backslash\{y_1,y_2,y_3,z_1\}$ and let $L'''(v)=L''(v)\backslash\{c_4,c_5\}$ for each vertex $v\in U'''$.  By Lemma~\ref{keylemma}, we may assume that $|L'''(U''')|<|U'''|=5$.  Since $|L'''(z_2)|+|L'''(z_3)|\geq 5$, vertices $z_2$ and $z_3$ have a common color, call it $c_6$.  Use color $c_6$ on $z_2$ and $z_3$, then color the remaining vertices greedily.
Thus, $G$ is $L$-colorable.
\end{proof}

\section{Discussion}
We believe that our methods can be extended to prove Ohba's conjecture for
more multipartite graphs with three parts each of size at least 3.
In particular, we suspect that our methods will be suitable to prove Ohba's
conjecture for
$K_{4*2,3,2*(k-7),1*4}$ and
$K_{4*3,2*(k-8),1*5}$.
Further, we believe that our methods will be sufficient to prove Ohba's conjecture for
$K_{3*4,2*(k-7),1*3}$.
\section{Acknowledgements}
We would like to express our gratitude to the referees for their careful reading and helpful comments.  %In particular, we are indebted to the referees for presenting Lemma~\ref{keylemma}, which helps shorten the proof of Theorem~\ref{theorem2}.

\end{document}